\documentclass[a4paper,11pt]{article}
\usepackage{bm}
\usepackage{amsmath}
\usepackage{amsthm}
\usepackage{wasysym}
\usepackage{amssymb}
\usepackage{amsfonts}
\usepackage{graphicx}
\usepackage[english]{babel}
\usepackage[utf8]{inputenc}
\usepackage[T1]{fontenc}
\usepackage{mathrsfs}
\usepackage{enumerate}
\usepackage{version}
\usepackage{calc}
\usepackage{subfigure}
\usepackage{authblk}
\usepackage{comment}

 \usepackage{pstricks,pst-math,pst-xkey}
 \usepackage{float}
 \usepackage{indentfirst}
\usepackage[pagewise]{lineno} 
\newtheorem{definition}{Definition}[section]
\newtheorem{theorem}[definition]{Theorem}
\newtheorem{lemma}[definition]{Lemma}
\newtheorem{corollary}[definition]{Corollary}
\newtheorem{proposition}[definition]{Proposition}
\newtheorem{remark}[definition]{Remark}

\newtheorem{example}[definition]{Example}

\newcommand{\R}{{\mathbb R}}

\newcommand{\ve}{{\varepsilon}}

\def\a{\alpha}

\def\s{\sigma}
\def\ve{\varepsilon}

\def\f{\varphi}

\def\t{\tau}

\def\F{\mathcal{F}}
\def\B{\mathcal{B}}

\def\f{\varphi}

\def\R{\mathbb{R}}
\def\E{\mathbb{E}}

\def\L{\mathcal{L}_2^0}

\begin{document}

\title{Invariant Measure for Neutral Stochastic Functional Differential Equations with Non-Lipschitz Coeffiecients}

\author[1]{ Andriy Stanzhytskyi\thanks{andrew.stanj@gmail.com}}

\author[2]{Oleksandr Stanzhytskyi \thanks{ostanzh@gmail.com}}

\author[3]{Oleksandr Misiats\thanks{omisiats@vcu.edu}}

\affil[1]{Department of Physics and Mathematics, Igor Sikorsky Kyiv Polytechnic Institute, Ukraine}

\affil[2]{Department of Mathematics,
Taras Shevchenko National University of Kyiv, Ukraine}

\affil[3]{Department of Mathematics, Virginia Commonwealth University,
Richmond, VA, 23284, USA}


\maketitle

\begin{abstract}
In this work we study the long time behavior of nonlinear stochastic functional-differential equations of neutral type in Hilbert spaces with non-Lipschitz nonlinearities. We establish the existence of invariant measures in the shift spaces for such equations. Our approach is based on Krylov-Bogoliubov theorem on the tightness of the family of measures. 


\end{abstract}
\section{Introduction}
In this work we study the asymptotic behaviour of the solutions of neutral type stochastic functional-differential equations of the form
\begin{align}\label{eq:1.1}
& d[u(t)+g(u_t)] = [Au + f(u_t)] dt + \sigma(u_t) d W(t) \text{ for }  t>0; &\\
\nonumber & u(t)= \f(t), t \in [-h,0), \ h>0.
\end{align}
Here $A$ is an inifinitesimal generator of a strong continuous semigroup $\{S(t), t\geq 0\}$ of bounded linear operators in real separable Hilbert space $H$. The noise $W(t)$ is a $Q$-Wiener process on a separable Hilbert space $K$. For any $h>0$ denote $C_h := C([-h,0],H)$ to be a space of continuous $H$-valued functions $\varphi:[-h,0] \to H$, equipped with the norm 
\[
\|\varphi\|_{C_h}:= \sup_{t \in [-h,0]} \|\varphi(t)\|_H,
\]
where $\|\cdot\|_H$ stands for the norm in $H$. Throughout this paper, $\|\cdot\|_H$ will be denoted with $\|\cdot\|$. The solution $u(t)$ of (\ref{eq:1.1}) is sometimes referred as a {\it state process}. We also denote $u_t := u(t+ \theta)$, where $\theta \in [-h,0)$. The functionals $f$ and $g$ map $C_h$ to $H$, and $\sigma: C_h \to \mathcal{L}_2^0$, where $\mathcal{L}_2^0 = \mathcal{L}(Q^{1/2}K,H)$ is the space of Hilbert-Schmidt operators from $Q^{1/2}K$ to $H$. Finally, $\varphi:[-h,0] \times \Omega \to H$ is the initial condition, where $(\Omega, \mathcal{F},P)$ is the probability space.

Differential equations of neutral type arise in modeling the memory effects in various dynamical processes (e.g. heat conduction). While the classic heat equation $u_t = \Delta u$  accurately describes the heat transfer
by conduction in many materials, this model has two important drawbacks. First, the classic heat equation does not take into account the memory effects which may be present in some materials, especially at low temperatures. Most importantly, the  heat equation predicts an unrealistic result: that a
thermal disturbance at one point of the body is instantly felt everywhere in the body,
although not uniformly. In other words in Fourier heat conductors, finite thermal discontinuities must propagate with infinite speed. These observations lead one to
believe that in materials with memory, the Fourier's law of heat conduction may be an approximation (perhaps for sufficiently steady temperature fields) to a more general nonlinear constitutive assumption relating the heat flux to the material's thermal history. Gurtin and Pipkin \cite{2} have proposed a memory theory of heat
conduction which is independent of the present value of the temperature gradient, generalizes constitutive relations deduced from kinetic theory by Maxwell \cite{32}
and Cattaneo \cite{31}, and has finite heat propagation speeds. The linearized version of Gurkin-Pipkin's model reads as
\[
\dot{u}(x,t) + \int_{0}^{\infty}\beta(s) \dot{u}(x,t - s) \,ds = C \Delta u(x,t),
\]
which is a particular example of  a neutral type differential equation. Similar memory effects emerge in Hodgkin-Huxley model, Dawson-Fleming model of population genetics \cite{5} and other models.

The natural first step in the analysis of  (\ref{eq:1.1}) is to establish the existence and uniqueness of solutions and their continuous dependence on the initial data. The question of well-posedness was addressed, e.g. in \cite{6} under the Lipschitz conditions. This condition was somewhat relaxed in \cite{7}. The work \cite{33} addressed similar questions for the equations with impulsive action, as well as with infinite delays. However, in our opinion, the proof of the main result in aforementioned works is not quite complete. The reason is that this proof is based on a successive approximation scheme, which is implicit. In particular, the existence of a solution at each iteration step is highly nontrivial and requires a separate proof. 

On the other hand, \cite{8} correctly establishes the existence and uniqueness of mild solution under the Lipschits and linear growth assumptions. Under the similar condition the works \cite{9, 10} establish the existence and uniqueness of solutions if the equation is perturbed with Brownian motion, as well as fractional Brownian motion. However, the continuous dependence on the initial conditions was not established there. 

To the best of our knowledge, the question of the existence of invariant measures for nonlinear equations of type (\ref{eq:1.1}) has not been addressed in literature. One exception is the work \cite{11}, which established the existence of a stationary solution for a linear equation of type (\ref{eq:1.1}). 

The main goal of this paper is to establish the existence of invariant measure for the equation (\ref{eq:1.1}) using Krylov-Bogoliubov theory on the tightness of a family of measures \cite{12}. More precisely, we will implement the compactness approach of Da Prato and Zabczyk \cite{13}. This approach was used in establishing the existence of invariant measure for a large class of stochastic partial differential equations with no delay \cite{14,15,16,17}. The results on long time behavior of ordinary differential equations of similar nature were obtained in \cite{29,30}. The existence of periodic solutions for neutral functional differential equations in deterministic case were derived in \cite{34}. Analogous results for non-neutral type stochastic functional differential equations were established in \cite{18}. 

In order to apply the compactness approach, we need to establish that the solutions are both Markovian and Feller. These properties were established for (\ref{eq:1.1}) in \cite{19} in case when $A$ is Laplacian. The analog of comparison theorem for such equations was established in \cite{20}. 

The paper is organized as follows. In Section \ref{sec2} we introduce the necessary notation and preliminary results, as well as formulate the main results. Section \ref{sec3} is devoted to the proof of the main result. In Section \ref{sec4} we discuss the application of the main results of the paper. 

\section{Preliminaries and Main Results}\label{sec2}

Throughout this paper $H$ and $K$ are Hibert spaces with norms $\| \cdot \|$ and $\| \cdot \|_K$. Let $(\Omega, \mathcal{F}, P)$ be a complete probability space, and $Q$ be linear bounded covariance operator  such that $tr(Q) < \infty.$ Introduce
\[
W(t):= \sum_{i=1}^{\infty} \sqrt{\lambda_i}\beta_i(t) e_i(x), \ t\geq 0,
\]
which is a $Q$-Wiener process on $t \geq 0$. Here $\beta_i(t)$ are standard, one dimensional, independent Wiener processes, $\{e_k, k\geq 1\}$ is an orthonormal system in $K$, and a sequence of nonegative real numbers $\lambda_k$ satisfying 
\[
Q e_k = \lambda_k e_k, k = 1,2,...
\] 
and
\[
\sum_{i=1}^{\infty} \lambda_i < \infty.
\]
Also let $\{F_{t}, t \geq 0\}$ be a normal filtration satisfying
\begin{itemize}
    \item $W(t)$ is $\mathcal{F}_t$-measurable;
    \item $W(t+h)-W(t)$ is independent of $\mathcal{F}_t$ $\forall h \geq 0, t\geq 0$.
\end{itemize}
As defined in the introduction, let
\[
\mathcal{L}_2^0 = \mathcal{L}_2(Q^{1/2}K,H)
\]
be the space of all Hilbert Schmidt operators from $Q^{1/2}K$ to $H$ with the dot product 
\[
\langle{\Phi, \Psi\rangle}_{\mathcal{L}_2^0} = {\rm tr}(\Phi Q \Psi^*).
\]
Let $A$ be an infinitesimal generator of an analytic semigroup $\{S(t), t \geq 0\}$ in $H$. 

{\bf Condition (H1).} If $\sigma(-A) \text{ is the spectrum of } (-A)$, we have
\[
 {\rm Re} \ \sigma(-A) > \delta >0, \ \text{ and } \ A^{-1} \text{ is compact in } H.
\]
It follows from \cite{21} that for $0 \leq \alpha \leq 1$ one can define the fractional power $(-A)^{\alpha}$, which is a closed linear operator with domain $D(-A)^{\alpha}$. Denote $H_{\alpha}$ to be the Banach space $D(-A)^{\alpha}$ endowed with the norm 
\[
\|u\|_{\alpha} :=\|(-A)^{\alpha} u\|,
\]
 which is equivalent to the graph norm of $(-A)^{\alpha}$. This way $H_0 = H$. It follows from \cite{22}, Sec. 1.4, that if $A^{-1}$ is compact, then $S(t)$ is compact for $t>0$.  Next, it follows from \cite{21}, Th. 3.2, p.48 that under the assumption {\bf (H1)} the semigroup $S(t)$ is continuous with respect to uniform operator topology for $t>0$. Thus, using  \cite{21}, Th. 3.3, p.48 we may conclude that the operator $A$ has a compact resolvent. 
 Consequently, using \cite{22}, Th. 1.4.8, we have the following result:
 \begin{lemma}\label{L:2.1} 
 Under the condition {\bf (H1)} the embedding $H_{\alpha} \subset H_{\beta}$ is compact if $0 \leq \beta < \alpha \leq 1$.
 \end{lemma}
   \begin{lemma}\label{L:2.2} (\cite{22}, Th. 1.4.3) Under the condition (H1), for every $\alpha \geq 0$ there exists $C_\alpha>0$ such that
   \[
   \|(-A)^{\alpha} S(t) \| \leq C_\alpha t^{-\alpha} e^{-\delta t}, t > 0.
   \]
In particular,
\begin{equation}\label{eq:exp_est}
  \|S(t) \| \leq C_0 e^{-\delta t}, t > 0.
  \end{equation}
  \end{lemma}
  \begin{lemma}\label{L:2.3}(\cite{13})
  Let $p>2$, $T>0$ and let $\Phi$ be an $\mathcal{L}_2^0$ valued, predictable process such that 
  \[
  \mathbb{E}  \int_0^T \|\Phi(t)\|^p_{\mathcal{L}_2^0} dt < \infty.
   \]
   Then there is a constant $M_T>0$ such that \[
   \mathbb{E} \sup_{t \in [0,T]} \left\|\int_0^t S(t-s) \Psi(s) d W(s) \right\|^p \leq M_T \mathbb{E} \int_0^T \| \Psi(s)\|^p_{\mathcal{L}_2^0} ds.
   \]
     \end{lemma}
   The following are the main assumptions on the coefficients of the equation (\ref{eq:1.1}):\\
   {\bf Condition (H2).} The mappings $f:C_h \to H$ and $\sigma: C_h \to \mathcal{L}_2^0$ are continuous and satisfy
   \begin{itemize}
       \item[{[i]}] There exists a positive constant $K>0$ such that 
       \[
       \|f(\f)\| + \|\s(\f)\|_{\L} \leq K(1+\|\f\|_C) \text{ for all } \f \in C_h.
       \]
       \item[{[ii]}] There is a  function $N:[0,+\infty) \to  [0,+\infty)$ such that
       \begin{itemize}
           \item[{(a)}] the function $N$ is continuous, nondecreasing, concaved, and $N(0)=0$;
           \item[{(b)}] For $p>2$ and 
           for all $\f_1, \f_2 \in C_h$ we have
           \[
           \|f(\f_1) - f(\f_2)\|^p + \|\s(\f_1) - \s(\f_2)\|^p_{\L} \leq N \left( \|\f_1 - \f_2\|^p_{C_h}\right).
           \]
        \item[{(c)}] The only nonnegative function $z(t)$ satisfying
        \[
        z(t) \leq D \int_{0}^t N(z(s)) ds \text{ for all } t \in [0,T]
        \]
        is $z \equiv 0$.
       \end{itemize}
       
   \end{itemize}
\begin{remark}\label{Rem1}
The condition (c) holds for any $N$ satisfying (a) and
\[
\int_{0}^1 \frac{ds}{N(s)} = +\infty.
\]
\end{remark}
{\bf Condition (H3).} There exist positive constants $0< \alpha \leq 1$ and $0< M_g < 1$ such that for all $\f_1, \f_2 \in C_h$ the function $g:C_h \to H_{\alpha}$ satisfies
\[
\|g(\f_1) - g(\f_2)\|_{H_\a} \leq M_g\|\f_1-\f_2\|_{C_h}.
\]
{\bf Condition (H4).} The initial condition $\f:[-h,0] \times \Omega \to H$ is an $\mathcal{F}_0$ - measurable random variable, which is independent of $W$ and has continuous trajectories. 

\begin{definition}\label{Def:mild}
\cite{6,7,8} A continuous $\mathcal{F}_t$-adapted stochastic process $u:[-h,T]\times \Omega \to H $ is a {\it mild solution for (\ref{eq:1.1})} for $t \in [0,T]$ if it satisfies the integral equation
\begin{equation}\label{DefMild}
u(t) = S(t)(\f(0) + g(\f)) - g(u_t)   - \int_0^t AS(t-s)g(u_s) ds  + \int_0^t S(t-s) f(u_s) ds + \int_0^t S(t-s) \s(u_s) d W(s)
\end{equation}
and $u(t) = \f(t)$ a.s. for $t \in [-h,0]$.
\end{definition}
The main results of this paper are the Theorem \ref{Th:2.1}, Theorem \ref{Th:2.2} and Theorem \ref{Th:2.4} below.
\begin{theorem}\label{Th:2.1}(Existence and uniqueness) Suppose the conditions {\bf (H1)-(H4)} hold. Then for all $T>0$ the equation (\ref{eq:1.1}) has a unique mild solution $u$ on $[0,T]$. Furthermore, this solution satisfies
\[
\|u_t - u_0\|_{C_h} = \sup_{\theta \in [-h,0]} \|u(t+ \theta) - \f(\theta)\| \to^P 0, t \to 0.
\]
\end{theorem}
The main idea of the proof of Theorem \ref{Th:2.1} is to consider an auxiliary semilinear equation, which, in turn, will enable us to constuct a convergent approximating sequence. In contrast with non-neutral type equations (with $g \equiv 0$), the primary difficulty of the equation (\ref{eq:1.1}) comes from the term
$AS(t-s)g(u_s)$. Generally speaking, if $g(u_s) \in H$, this term has a non-integrable singularity at $t=s$. We overcome this difficulty by introducing the fractional powers of the operator $(-A)$. In particular, we show that if $g$ is regular enough for $g(u_s) \in D(-A)^{\alpha}$ for $\alpha > 0$, then this singularity becomes integrable. 

 \begin{remark}
In order to prove Theorem \ref{Th:2.1} we will use the methods outlined in \cite{24}, which established the corresponding result for the neutral type stochastic equations in a simpler case.  Besides, Theorem  \ref{Th:2.1} has a different condition on the Lipschitz constant for $g$. Namely,  \cite{24} requires 
\[
4^{p-1} \|(-A)^{-\a}\|^p M_g < 1,
\]
while we require a much more easy to check condition $M_g<1$. Similar conditions {\red on  $(-A)^{-\alpha}$} appear in \cite{7,33} among others. In addition, we relax the regularity condition for the map $g(\f)$. In particular, we require $g \in H_\alpha$ for $\alpha  \in (0,1)$, while in \cite{24} the authors require $\alpha \in (\frac{1}{p}, 1)$ for $p>2$. 
\end{remark}

\begin{theorem}\label{Th:2.2}(Continuous dependence on the initial data) Suppose the conditions of Theorem \ref{Th:2.1} hold. Let $u(t,\f)$ and $u(t,\Psi)$ be two solutions of (\ref{eq:1.1}) with the initial conditions $\f$ and $\Psi$ satisfying the condition {\bf (H4)}. Then
\begin{equation*}\label{cont_dep}
\E \sup_{t \in [0,T]} \left[\|u(t,\f) - u(t, \Psi)\|^p +  \|u_t(\f) -u_t(\Psi)\|_{C_h}^p\right] \to 0,  
\end{equation*}
as $\E \|\f - \Psi\|_{C_h}^p \to 0.$
\end{theorem}

Denote $B_b(C_h)$ to be the Banach space of bounded real Borel measurable functionals, on $C_h$, and $C_b(C_h)$ is the space of bounded continuous functionals on $C_h$. Since Theorem \ref{Th:2.1} guarantees the existence and uniqueness of the solution at any $t \geq 0$, replacing the initial interval $[-h,0]$ with $[-h+s,s]$ for any $s \geq 0$ we can guarantee the existence and uniqueness of the solution on any interval $[s,t]$ with the initial $\mathcal{F}_s$ measurable function $\f$ satisfying the conditions on $[s-h,s]$. Such solutions will be denoted with $u(t,s,\f)$. Similarly, denote $u_t(s,\theta) = u(t+\theta, s, \f), \theta \in [-h,0]$ to be the shift of the solution $u(t,\f)$, such that $u_s(s,\f) = u(s+\theta,s,\f) = \f(\theta)$.

Following \cite{23}, define the family of shift operators
\begin{equation}\label{eq:2.3}
U_s^t \f:= u(t+ \theta, s, \f) = u_t(s,\f).
\end{equation}
Denote $\mathcal{F}_s^t (dW)$ to be the minimal $\sigma$ - algebra containing $W(\tau) - W(s), \tau \in [s,t]$, and $G^t$ to be the minimal $\sigma$ - algebra containing $W(\tau) - W(t)$ for $\tau \geq t$.  For any nonrandom $\varphi \in C_h$ with $s \geq 0$ and $t \geq s$ the solution $U_s^t \f = u_t(s,\f)$ is a $\mathcal{F}_{s}^t(d W)$- measurable random function taking values in $C_h$.   Note that in this case, since $\f$ is nonrandom, $u_t(s,\f)$ is independent of the $\sigma$ - algebra $G^t$. The next proposition was shown, e.g. in \cite{18}:
\begin{proposition}
The operator (\ref{eq:2.3}) satisfies
\[
U_{\tau}^t U_{s}^{\tau} \f = U_s^t \f
\]
for all $t \geq \tau \geq s \geq 0$, and $\f \in C_h$.
\end{proposition}

Let $D$ be a $\sigma$-algebra of Borel subsets of $C_h$. For every $A \in D$ we define
\begin{equation}\label{def of P}
\mu_t(A)  = P\{u_t(s,\varphi) \in A\} = P\{U_s^t \varphi  \in A\} = P(s,\varphi,t,A).
\end{equation}
This way $u_t(s,\f)$ naturally defines a measure on $D$. The formula (\ref{def of P}) defines a transition probability function, corresponding to the random process $u_t(s,\f), t \geq s \geq 0$. In a similar way to the finite dimensional case \cite{23}, it is possible to show that this function satisfies the properties of the transition probability. This way we have
\begin{theorem}
(Markov property) Under the assumptions of Theorem \ref{Th:2.1}, the process $u_{t}(s,\f)$ is the Markov process on $C_h$ with the transition function, defined by (\ref{def of P}).
\end{theorem}
\begin{proposition}\label{Prop:2.3}
For any $t \geq s \geq 0$ and any $A \in D$ we have
\begin{equation*}
P(s,\f,t,A) = P(0,\f,t-s,A).
\end{equation*}
\end{proposition}
\begin{proof}
Let $\tilde{u}(t) = u(s+t,s,\f)$. Then $\tilde{u}_0 = u(s+\theta, s, \f), \tilde{u}_t = u(s+t+\theta, s, \f) = u_{s+t}(s,\f)$, for $\theta \in [-h,0].$ On the other hand, 
\begin{eqnarray*}
    & &\tilde{u}(t) = u(s+t,s,\f) = S(t)(\f(0)+ g(\f)) - g(u_{t+s}) - \int_{s}^{s+t} A S(t+s-\tau) g(u_\tau) d\tau \\
   & &+  \int_s^{s+t} S(t+s-\tau) f(u_\tau) d \tau +  \int_s^{s+t} S(t+s-\tau) \sigma(u_\tau) d W(\tau) = S(t)(\f(0)+ g(\f)) - g(u_{t+s}) \\
   & & - \int_0^t A S(t-\tau) g(u_{\tau+s}) d\tau + \int_0^t  S(t-\tau) f(u_{\tau+s}) d\tau + \int_0^t S(t-\tau) \sigma(u_{\tau+s}) d \tilde{W}(\tau),
\end{eqnarray*}
where
$\tilde{W}(\tau) := W(s+ \tau) - W(s)$ is a $Q$-Wiener process as well. This way
\begin{eqnarray}\label{2.6}
    & &\tilde{u}(t) = S(t)(\f(0)+ g(\f)) - g(\tilde{u}_{t}) - \int_{0}^{t} A S(t-\tau) g(\tilde{u}_\tau) d\tau \\
   \nonumber & & +  \int_0^{t} S(t-\tau) f(\tilde{u}_\tau) d \tau +  \int_0^{t} S(t-\tau)  \sigma(\tilde{u}_\tau) d \tilde{W}(\tau).
\end{eqnarray}
However, the same equation is satisfied with $u(t,0,\f)$, with $u(0,0,\f) = \f(0)$ and $u_0 = \f(\theta)$. The only difference is that $u(t,0,\f)$ solves (\ref{2.6}) with a different $Q$-Wiener process $\tilde{W}$. However, since the distribution function for $W$ is the same as for $\tilde{W}$, the uniqueness of solution implies that the distributions of $u(s+t, s,\f)$ and $u(t,0,\f)$ coincide. In other words, the distribution of $u(s+t, s,\f)$ is independent of $s$. Then the distribution of $u_t(s,\f) = u(t + \theta, s, \f) = u(t - s + \theta + s, s, \f)$ matches the distribution of $u(t - s + \theta, 0, \f) = u_{t-s}(0,\f)$. Hence
\begin{eqnarray*}
P(s,\f,t,A) = P\{u_t(s,\f) \in A)\} & = & P\{u(t+\theta, s, \f) \in A\} \\
& = & P\{u(t+\theta -s, 0, \f) \in A\} = P\{u_{t-s}(0,\f) \in A\},
\end{eqnarray*}
which yields the desired result.
\end{proof}
For $g \in B_b(C_h)$, for all $\f \in C_h$ and $t \geq s \geq 0$ define 
\[
P_{s,t}(\f) := \mathbb{E} \, g(u_t(s,\f)).
\]
From Proposition \ref{Prop:2.3} we have $P_{s,t}(\f) = P_{0,t-s}(\f)$ and denote $P_t \f = P_{0,t}(\f)$. The following proposition follows from Theorem \ref{Th:2.2}.
\begin{proposition}
Under the assumptions {\bf (H1)-(H4)} the transition semigroup $P_t, t\geq 0$ is stochastically continuous and and satisfies the Feller property 
\[
P_t: C_b(C_h) \to C_b(C_h) \text{ and } \lim_{t \to 0} P_t \f = \f. 
\]
\end{proposition}
We are now in position to state our main result on the long time behavior of solutions of (\ref{eq:1.1}).
\begin{theorem}\label{Th:2.4}
Assume the assumptions of Theorem \ref{Th:2.1} hold. Suppose the equation (\ref{eq:1.1}) has a solution $u(t)$ which is bounded in probability for $t \geq 0$ in $C_h$, namely
\begin{equation}\label{eq:2.7}
\sup_{t \geq 0} P\{\|u_t\|_{C_h} > R\} \to 0, R \to \infty.
\end{equation}
Then there exists an invariant measure $\mu$ in $C_h$, i.e.
\[
\int_{C_h} P_t \f(x) d \mu(x) = \int_{C_h} \f(x) d \mu(x), \text{ for any } t \geq 0, \text{ and } \f \in C_b(C_h).
\]
\end{theorem}

\begin{remark}
The condition (\ref{eq:2.7}) is a standard condition for the existence of invariant measures, see e.g. \cite{25}. In Section \ref{sec4} we will show how to verify this condition in terms of the coefficients of a neutral type stochastic functional-differential equation with elliptic operator.
\end{remark}

\section{Proofs of the main results}\label{sec3}
\subsection{Proof of existence and uniqueness}
In this subsection we will prove Theorem \ref{Th:2.1}. We are going to use the approach from \cite{24}. Assume first that $\E \sup_{\theta \in [-h,0]} \|\f(\theta)\|^p < \infty, p>2$. Let $\mathcal{L}_{\mathcal{F}}^p([0,T],H)$ be the Banach space of $H$-valued, $\mathcal{F}_t$-adapted random processes with the norm
\[
\|\zeta\|^p_{H,T} := \E \int_0^T \|\zeta(t)\|^p dt.
\]
Similarly, let $\mathcal{L}_{\mathcal{F}}^p([0,T],\mathcal{L}_2^0)$, be the Banach space of $\mathcal{L}_2^0$ valued $\mathcal{F}_t$-adapted random processes $\Phi(t)$ with the norm
\[
\|\Phi\|^p_{L_2^0, T}:=\E \int_{0}^T \|\Phi(t)\|_{L_2^0}^p dt.
\]
Next, denote $\mathcal{B}_{p,T}$ to be the Banach space of all $H$-valued, measurable, continuous $\mathcal{F}_t$-adapted (for $t \geq 0$) processes $u(t,\omega):[-h,T] \times \Omega \to H$, equipped with the norm \begin{equation*}
    \|u\|_{\mathcal{B}_{p,T}}^p := \E \sup_{t \in [-h,T]} \|u(t)\|^p.
\end{equation*} 
Finally, introduce the closed set
\[
\mathcal{B}_{p,T}(\f):=\{u \in \mathcal{B}_{p,T}, u(t) = \f(t) \text{ for } t \in [-h,0]\}.
\]
We will use the following Lemma:
\begin{lemma}\label{L:2.4}
For $T>0$, let $u(t), -h \leq t \leq T$ be a stochastic process with continuous paths. Then, if $p \geq 1$ and $u_0 = \f$, we have
\[
\E \sup_{t \in [0,T]} \|u_t\|_{C_h}^p \leq \E \|\f\|_{C_h}^p +\E \sup_{t \in [0,T]} \|u(t)\|^p.
\]
\end{lemma}
The proof of Lemma \ref{L:2.4} follows directly from the definition of $u_t$. We next split the proof into the following steps:\\
 {\bf{Step 1.}} We start with an auxiliary equation. For any $(a,b) \in \mathcal{L}_{\mathcal{F}}^p([0,T],H) \times \mathcal{L}_{\mathcal{F}}^p([0,T],L_2^0)$ consider the following equation
 \begin{equation}\label{eq:3.1}
 \begin{cases}
     d(u(t) + g(u_t)) = [A u(t) + a(t)] dt +b(t) d W(t), t \geq 0;\\
     u(t) = \f(t), t \in [-h,0].
\end{cases}
 \end{equation}
 \begin{lemma}\label{lem:3.1}
 Under the condition {\bf (H3)} the equation (\ref{eq:3.1}) has a unique mild solution in $\mathcal{B}_{p,T}(\f)$.
 \end{lemma}
 \begin{proof}
 Define the following operator $\Psi: \mathcal{B}_{p,T}(\f) \to \mathcal{B}_{p,T}(\f)$ as
 \begin{equation}\label{eq:3.2}
 [\Psi u](t):=S(t) (\f(0)+g(\f)) -  g(u_t) - \int_{0}^t A  S(t-s) g(u_s) ds + \int_{0}^t   S(t-s) a(s) ds + \int_0^t S(t-s) b(s) dW(s), 
 \end{equation}
 defined for $t \in [0,T]$, with $u(t) = \f(t)$ for $t \in [-h,0]$. Let us show that $\Psi:\mathcal{B}_{p,T}(\f) \to \mathcal{B}_{p,T}(\f)$ is a contraction. To this end, let us first show that $\Psi$ takes values in $\mathcal{B}_{p,T}(\f)$. Using Lemma \ref{L:2.4}  we have
 \[
 \E \sup_{t \in [0,T]} \|g(u_t)\|^p \leq \E \sup_{t \in [0,T]} \|g(u_t)\|^p_{\a} \leq M_g^p (\|g(0)\|_{C_h}^p +\E \sup_{t \in [0,T]} \|u(t)\|^p).
 \]
 Next,
 \begin{eqnarray*}
 & & \E \sup_{t \in [0,T]} \left\|\int_0^t A S(t-s) g(u_s) ds\right\|^p = \E \sup_{t \in [0,T]} \left\|\int_0^t (-A)^{1-\a} S(t-s)(-A)^{\a} g(u_s) ds\right\|^p \\
 & &  
 \leq \E \sup_{t \in [0,T]} \left( \int_0^t \left\|(-A)^{1-\a} S(t-s) \right\| \left\|(-A)^{\a} g(u_s) \right\| ds\right)^p \\
 & &\leq \sup_{t \in [0,T]} \left(\int_0^t C_{1-\alpha} (t-s)^{-(1-\alpha)} ds \right)^p \E \sup_{t \in [0,T]} \|g(u_t)\|^p_{\a}.
 \end{eqnarray*}
The remaining two terms can be estimated in a standard way (see, e.g. \cite{24}):
\begin{equation*}\label{est1}
    \E \sup_{t \in [0,T]} \left\| \int_0^t  S(t-s) a(s) ds \right\|^p \leq C_0^p \|a\|^p_{H,T}
\end{equation*}
and
\begin{equation*}\label{est2}
\E \sup_{t \in [0,T]} \left\| \int_0^t  S(t-s) b(s) d W(s) \right\|^p \leq C_0^p \|b\|^p_{L_2^0,T},
\end{equation*}
where $C_0$ is given by (\ref{eq:exp_est}).
 
 Let us now show that $\Psi$ is contraction. Fix $T_1 \in [0,T]$  For any $\rho \in (0,1)$, and $a, b\geq 0$ we will use the inequality 
 \[
 (a+b)^p \leq \frac{a^p}{\rho^{p-1}} + \frac{b^p}{(1-\rho)^{p-1}}.
 \]
 For any $u,v \in B_{p,T_1}$, we have
 \begin{eqnarray*}
 & &\left\| \Psi(u) - \Psi(v) \right\|^p_{B_{p,T_1}} \leq \frac{1}{M_g^{p-1}}\|g(u_t) - g(v_t)\|^p_{B_{p,T_1}} \\
  & & + \frac{1}{(1-M_g)^{p-1}}\left\|\int_0^t A S(t-s)(g(u_s)-g(v_s)) ds \right\|^p_{B_{p,T_1}} 
\leq \frac{1}{M_g^{p-1}}\E \sup_{t \in [0,T_1]}\|g(u_t) - g(v_t)\|^p  \\
 & & + \frac{1}{(1-M_g)^{p-1}} \E \sup_{t \in [0,T_1]} \left(\int_0^t \|(-A)^{1-\a} S(t-s)\| \|(-A)^\a (g(u_s)-g(v_s)\| ds\right)^p \\
& &\leq M_g \E \sup_{t \in [0,T_1]} \|u_t-v_t\|^p_{C_h} + 
\frac{M_g^p C_{1-\alpha}^p T_1^{\a p}}{(1-M_g)^{p-1} \a^p}  \E \sup_{t\in [0,T_1]}\|u_t-v_t\|^p_{C_h} = \gamma(T_1) \E \sup_{t\in [0,T_1]}\|u_t-v_t\|^p_{C_h}.
 \end{eqnarray*}
 Here we used Lemma \ref{L:2.4} and Lemma \ref{L:2.2}. Using the assumption {\bf (H3)} we can choose $T_1$ such that
 \begin{equation}\label{cond1}
 \gamma(T_1):= M_g +  \frac{M_g^p  C_{1-\alpha}^p T_1^{\a p}}{(1-M_g)^{p-1} \a^p} < 1,
 \end{equation}
 hence the operator $\Psi$ is a contraction in $\mathcal{B}_{p,T_1}(\f)$. Therefore $\Psi$ has a unique fixed point, which is a unique fixed point, which is a mild solution of (\ref{eq:3.1}) on $[0,T_1]$. 
 
 Let us show that the solution has a continuous trajectory. It is straightforward to verify the continuity of the first, third and fourth terms in the equation (\ref{eq:3.2}). The continuity of the fifth term follows from the factorization method \cite{13}, Prop. 7.3. 
 
 Let us now establish the continuity of the second term. Indeed, since $u$ is continuous,
 \begin{equation*}
 \|g(u_{t+s}) - g(u_t)\| \leq M_g \|u_{t+s} - u_t\|_{C_h} =  \sup_{\theta \in [-h,0]} \|u(t+s+\theta) - u(t+\theta)\| \to 0, s \to 0 \ a.s.
 \end{equation*}
 We may now consider the new initial value problem with the initial condition on $[T_1-h, T_1]$ and so on. This way the solution may be extended to the entire interval $[0,T]$ in  finitely many steps. Hence the equation (\ref{eq:3.1}) has a unique mild solution, defined on $[0,T]$, which completes the proof of the Lemma. 
 \end{proof}
 {\bf Step 2.} For given $\f \in C_h$ we define the operator 
 \[
 \Phi_{\f}: L^{p}_{\F}([0,T],H) \times L^{p}_{\F}([0,T],L_2^0) \to \B_{p,T}
 \]
 as 
 \begin{equation*}
 (\Phi_{\f}(a,b))(t) := 
 \begin{cases}
 \f(t), t \in [-h,0];\\
 u(t), t \in [0,T],
 \end{cases}
 \end{equation*}
 where $u$ is the unique solution of (\ref{eq:3.1}).

 \begin{lemma}\label{lem:3.2}
 Under the assumption {\bf (H3)}, for every $p > 2$, there are positive constants $B$, $D$ and $L$, depending only on $T$, $M_g$, $p$ and $\alpha$, such that for any $t \in [0,T]$, for any $\f$ and $\psi$ in $C_h$, and for all pairs 
 \[
 (a,b), (a_1,b_1) \text{ and } (a_2,b_2) \in  L^{p}_{\F}([0,T],H) \times L^{p}_{\F}([0,T],L_2^0), 
 \]
 we have
 \begin{eqnarray}
 \label{eq:3.3} & &\|\Phi_{\f}(a_1,b_1) - \Phi_{\psi}(a_2,b_2)\|^p_{\B_{p,t}} \leq L \|\f -\psi\|^p_{C_h}  \\
\nonumber  & & +  B \int_0^t\left[\E \|a_1(s)-a_2(s)\|^p + \E \|b_1(s)-b_2(s)\|^p_{L_2^0}\right] ds, \ \ \text{ and }\\
 \label{eq:3.4} & & 
 \|\Phi(a,b)\|^p_{\B_{p,t}} \leq D + B\int_0^t\left[\E \|a(s)\|^p + \E\|b(s)\|^p_{L_2^0}\right] ds.
 \end{eqnarray}
 \end{lemma}
 \begin{proof} For any $(a_1,b_1)$ and  $(a_2,b_2)$ we have
 \begin{multline*}
\Phi_{\f}(a_1,b_1)  := u = S(t) (\f(0)+g(\f)) -  g(u_t) - \int_{0}^t A  S(t-s)g(u_s) ds \\
+\int_{0}^t S(t-s) a_1(s) ds + \int_0^t  S(t-s) b_1(s) dW(s),
 \end{multline*}
 and
 \begin{multline*}
\Phi_{\psi}(a_2,b_2)  := v = S(t) (\psi(0)+g(\psi)) -  g(v_t) - \int_{0}^t A  S(t-s)g(v_s) ds \\
+ \int_{0}^t   S(t-s) a_2(s) ds + \int_0^t  S(t-s) b_2(s) dW(s).
 \end{multline*}
 Denote
 \[
 \beta(s) := \E \|a_1(s)-a_2(s)\|^p + \E \|b_1(s)-b_2(s)\|^p_{L_2^0}.
 \]
 Now choose $T_1>0$, which satisfies both (\ref{cond1}) and
 \begin{equation}\label{cond2}
   M_g + \left(\frac{5}{1-M_g}\right)^{p-1}\left(\frac{C_{1-\alpha} T_1^\a M_g}{\a}\right)^p<1. \end{equation}
Note that such $T_1$ always exists provided $M_g<1$. Next, we partition the interval $[0,T]$ with the points $kT_1$, and consider the interval $t \in [0,T_1]$ first. For such $t$ we have
\begin{eqnarray}\label{eq:3.6}
\nonumber & &\|u-v\|_{\B_{p,t}}^p = \E \sup_{s\in[0,t]}\|u(s)-v(s)\|^p \leq \E \sup_{s \in [0,t]} \Big[\|S(s)(g(\f) - g(\psi))\|  + \|S(s)(\f(0) - \psi(0))\| \\
\nonumber & & + \|g(u_s)-g(v_s)\| + \left\|\int_0^s A S(s-\tau)(g(u_\tau) - g(v_\tau)) d\tau \right\| +  \left\|\int_0^s S(s-\tau) (a_1(\tau) - a_2(\tau)) d\tau \right\| \\
& & + \left\|\int_0^s S(s-\tau) (b_1(\tau) - b_2(\tau)) dW(\tau) \right\|\Big]^p \leq \frac{1}{M_g^{p-1}}\E \sup_{s \in [0,t]}\|g(u_s) - g(v_s)\|^p \\
\nonumber & & + \left(\frac{5}{1-M_g}\right)^{p-1} C_0^p(1 + M_g^p) \|\f-\psi\|_{C_h}^p  \\
\nonumber & & + \left(\frac{5}{1-M_g}\right)^{p-1} \E \sup_{s \in [0,t]} \left(\int_0^s \|(-A)^{1-\a} S(s-\t)\| \|(g(u_\t)-g(v_\t)\|_{\a} d\t \right)^p  \\
\nonumber & &
+C_{0}^p \left(\frac{5}{1-M_g}\right)^{p-1} \int_0^t \beta(s) ds,
\end{eqnarray}
where $C_0$ is given in  (\ref{eq:exp_est}).
Let us estimate each term in (\ref{eq:3.6}) separately. 
\begin{eqnarray*}\label{eq:3.7}
 \nonumber  & & \E \sup_{s \in [0,t]} \|g(u_s)-g(v_s)\|^p \leq  \E \sup_{s \in [0,t]} \|g(u_s)-g(v_s)\|^p_\a \leq M_g^p \E  \sup_{s \in [0,t]} \|u_s-v_s\|^p_{C_h} \\
  & & = M_g^p \E \sup_{s \in [0,t]} \sup_{\theta \in [-h,0]} \|u(s+\theta)- v(s+\theta))\|^p =  M_g^p \E \sup_{s \in [-h,t]}\|u(s)-v(s)\|^p\\
 \nonumber & & \leq  M_g^p \|\f - \psi\|_{C_h}^p + M_g^p \E \sup_{s \in [0,t]}\|u(s)-v(s)\|^p.
\end{eqnarray*}
Next,
\begin{multline}\label{eq:3.8}
\E \sup_{s \in [0,t]} \left(\int_0^s \|(-A)^{1-\a} S(s-\t)\| \|(g(u_\t)-g(v_\t)\|_{\a} d\t \right)^p \leq \left(\frac{M_g C_{1-\a} T_1^\a}{\a}\right)^p \E \sup_{s \in [-h,t]}\|u(s)-v(s)\|^p\\
\leq \left(\frac{M_g C_{1-\a} T_1^\a}{\a}\right)^p \left(\E \sup_{s \in [0,t]}\|u(s)-v(s)\|^p + \|\f - \psi\|_{C_h}^p \right).
\end{multline}
Using (\ref{eq:3.6}) - (\ref{eq:3.8}), we have
\begin{multline*}
\|u-v\|^p_{\B_{p,t}} \leq M_g \E \sup_{s \in [0,t]} \|u(s)-v(s)\|^p + \left(\frac{5}{1-M_g}\right)^{p-1} \left(\frac{M_g C_{1-\a} T_1^\a}{\a}\right)^p \E \sup_{s \in [0,t]} \|u(s)-v(s)\|^p  \\
+ \left(\frac{5}{1-M_g}\right)^{p-1} C_{0}^p \int_{0}^{t}\beta(s)ds + \tilde{L}_1 \|\f-\psi\|^p_{C_h},
\end{multline*}
or
\begin{equation}\label{eq:3.9}
 \|u-v\|^p_{\B_{p,t}} \leq C_{T_1}  \int_0^t \beta(s) ds + L_1 \|\f-\psi\|^p_{C_h},
\end{equation}
where 
\[
C_{T_1} = \frac{\left(\frac{5}{1-M_g}\right)^{p-1}}{1-M_g - \left(\frac{5}{1-M_g}\right)^{p-1} \left(\frac{M_g C_{1-\a} T_1^\a}{\a}\right)^p} C_{0}^p
\]
and
\[
L_1 = \left(\frac{5}{1-M_g}\right)^{p-1}\frac{C_0^p + M_g^p C_0^p + M_g^p +  \left(\frac{M_g C_{1-\a} T_1^\a}{\a}\right)^p}{1-M_g - \left(\frac{5}{1-M_g}\right)^{p-1} \left(\frac{M_g C_{1-\a} T_1^\a}{\a}\right)^p}.
\]

Let us now consider the second interval $t \in [T_1, 2 T_1]$. We have
 \begin{multline} \label{eq:3.10}
u(t) = S(t-T_1) (u(T_1)+g(u_{T_1})) -  g(u_t) - \int_{T_1}^t A  S(t-s) g(u_s) ds \\
+\int_{T_1}^t  S(t-s) a_1(s) ds + \int_{T_1}^t  S(t-s) b_1(s) dW(s),
 \end{multline}
 and
 \begin{multline*}
v(t) = S(t-T_1) (v(T_1)+g(v_{T_1})) -  g(v_t) - \int_{T_1}^t A  S(t-s)g(v_t) ds \\
+ \int_{T_1}^t   S(t-s) a_2(s) ds + \int_{T_1}^t  S(t-s) b_2(s) dW(s).
 \end{multline*}
 Then
 
 \begin{eqnarray}\label{eq:3.11}
\nonumber \|u-v\|_{\B_{p,t}}^p = & &  \E \sup_{s \in [0,t]}\|u(s) - v(s)\|^p   \\
& &= \max\{\E \sup_{s \in [0,T_1]} \|u(s) - v(s)\|^p, \E \sup_{s \in [T_1,t]} \|u(s) - v(s)\|^p\} \\
\nonumber & &\leq \E \sup_{s \in [0,T_1]} \|u(s) - v(s)\|^p + \E \sup_{s \in [T_1,t]} \|u(s) - v(s)\|^p.
 \end{eqnarray}
 Let us now estimate the second term in (\ref{eq:3.11}). From (\ref{eq:3.10}) we have
 \begin{eqnarray}\label{eq:3.12}
  \nonumber & &\E \sup_{s \in [T_1,t]} \|u(s) - v(s)\|^p \leq \E \sup_{s \in [T_1,t]}\Big( \|g(u_s) - g(v_s)\| + \|S(s-T_1)\| (\|u(T_1)-v(T_1)\| \\ 
  \nonumber & &+\|g(u_{T_1}) - g(v_{T_1})\|) + \int_{T_1}^s C_{1-\a}(s-\t)^{-(1-\a)} M_g \|u_{\t} - v_{\t}\|_{C_h} d \t + \\ & &\left\|\int_{T_1}^s S(s-\t)(a_1(\t)-a_2(\t)) d\t \right\| + \left\|\int_{T_1}^s S(s-\t)(b_1(\t)-b_2(\t)) dW(\t) \right\|\Bigg)^p  \\
  \nonumber & & \leq \frac{1}{M_g^{p-1}} \E \sup_{s \in [T_1,t]} \|g(u_s) - g(v_s)\|^p + \frac{5^{p-1}}{(1-M_g)^{p-1}} \Big( C_{0}^p \Big(\E \|u(T_1) - v(T_1)\|^p \\
  \nonumber & & \E \|g(u_{T_1}) - g(v_{T_1})\|^p\Big)\Big) + \E \sup_{s \in [T_1,t]} \left(\int_{T_1}^s C_{1-\a}(s-\t)^{-(1-\a)} M_g \|u_{\t} - v_{\t}\|_{C_h} d \t \right)^p + C_{0}^p \int_{T_1}^t \beta(\t) d\t.
 \end{eqnarray}
 Setting $t = T_1$ in (\ref{eq:3.9}), we have
 \begin{equation*}\label{eq:3.13}
  \E \|u(T_1)-v(T_1)\|^p \leq C_{T_1} \int_0^{T_1} \beta(s) ds + L_1 \|\f- \psi\|_{C_h}^p
 \end{equation*}
 and
 \begin{multline*}\label{eq:3.14}
    \E \|g(u_{T_1})-g(v_{T_1})\|^p \leq  M_g^p \E \sup_{\theta \in [-h,0]}  \|u(T_1+\theta) - v(T_1 + \theta)\|^p  \\
    \leq M_g^p \E \sup_{s \in [-h,T_1]} \|u(s)-v(s)\|^p \leq M_g^p C_{T_1} \int_0^{T_1} \beta(s) ds + L_{21} \|\f- \psi\|_{C_h}^p .
 \end{multline*}
 Using the result of Lemma \ref{L:2.4}, we proceed with the following estimate:
 \begin{eqnarray*}\label{eq:3.15}
   \nonumber & & \E \sup_{s \in [T_1,t]} \|g(u_{s})-g(v_{s})\|^p \leq  M_g^p \E \sup_{s \in [T_1,t]} \sup_{\theta \in [-h,0]} \|u(s+\theta)- v(s+\theta)\|^p \\
   & &\leq M_g^p \E \sup_{ s \in [T_1-h,t]}\|u(s)- v(s)\|^p \leq M_g^p \max\{ \E \sup_{s \in [0,T_1]}\|u(s)- v(s)\|^p,  \E \sup_{s \in [T_1,t]}\|u(s)- v(s)\|^p\}  \\
    \nonumber & & \leq M_g^p C_{T_1} \int_0^{T_1} \beta(s) ds + M_g^p \E \sup_{s \in [T_1,t]}\|u(s)- v(s)\|^p + M_g^p L_1 \|\f-\psi\|_{C_h}^p.
 \end{eqnarray*}
 We next estimate the integral term in (\ref{eq:3.12}). Using Lemma \ref{L:2.4} once again, we have

 \[
 \E \sup_{s \in [T_1,t]} \left(\int_{T_1}^s \frac{C_{1-\a}}{(s-\t)^{1-\a}} M_g \|u_{\t} - v_{\t}\|_{C_h} d \t \right)^p
 \]
 \[
 \leq \left(\frac{M_g C_{1-\a} T_1^\a}{\a}\right)^p \ \E \sup_{s \in [T_1,t]} ( \sup_{\t \in [T_1,s]} ( \sup_{\theta \in [-h,0]} \|u(\t +\theta) - v(\t + \theta)\|^p)) 
 \]
 \begin{equation}\label{eq:3.16}
     =  \left(\frac{M_g C_{1-\a} T_1^\a}{\a}\right)^p \E \sup_{s \in [T_1-h, t]} \|u(s)- v(s)\|^p \leq \left(\frac{M_g C_{1-\a} T_1^\a}{\a}\right)^p C_{T_1} \int_0^{T_1} \beta(s) ds 
 \end{equation}
 \[
 + \left(\frac{M_g C_{1-\a} T_1^\a}{\a}\right)^p\ \E \sup_{s \in [T_1, t]} \|u(s)- v(s)\|^p + L_{22}\|\f-\psi\|^p_{C_h}.
 \]

It follows from (\ref{eq:3.12})-(\ref{eq:3.16}) that
\begin{eqnarray*}
 & &\E \sup_{s \in [T_1, t]} \|u(s)- v(s)\|^p \leq M_g C_{T_1} \int_0^{T_1} \beta(s) ds + M_g \E \sup_{s \in [T_1, t]} \|u(s)- v(s)\|^p  \\
 & & \left(\frac{5}{1-M_g}\right)^{p-1}
 C_{0}^p C_{T_1} \int_0^{T_1} \beta(s) ds + \left(\frac{5}{1-M_g}\right)^{p-1} C_{0}^p C_{T_1} M_g^p \int_0^{T_1} \beta(s) ds  \\ 
 & & +\left(\frac{5}{1-M_g}\right)^{p-1} \left(\frac{M_g C_{1-\a} T_1^\a}{\a}\right)^p \E \sup_{s \in [T_1, t]} \|u(s)- v(s)\|^p  \\
 & & + \left(\frac{5}{1-M_g}\right)^{p-1}  C_0^p \int_{T_1}^t \beta(s) ds + L_{23} \|\f-\psi\|_{C_h}^p,  
 \end{eqnarray*}
 or,
 \begin{eqnarray*}
 & &\E \sup_{s \in [T_1, t]} \|u(s)- v(s)\|^p \left(1- M_g - \left(\frac{5}{1-M_g}\right)^{p-1} \left(\frac{M_g C_{1-\a} T_1^\a}{\a}\right)^p \right)   \\
 & & \leq \left(M_g C_{T_1} +  \left(\frac{5}{1-M_g}\right)^{p-1} C_0^p C_{T_1}  + \left(\frac{5}{1-M_g}\right)^{p-1} C_{0}^p C_{T_1} M_g^p \right) \int_0^{T_1}  \beta(s) ds \\
 & & +  \left(\frac{5}{1-M_g}\right)^{p-1} C_0^p \int_{T_1}^t  \beta(s) ds + L_{23} \|\f-\psi\|_{C_h}^p.
 \end{eqnarray*}
 In view of (\ref{eq:3.11}), we conclude that
 \[
 \|u-v\|_{\B_{p,t}}^p \leq C_2(T_1) \int_{0}^{t} \beta(s) ds + L_2(T_1) \|\f-\psi\|_{C_h}^p,
 \]
 where
 \[
 C_2(T_1) = \frac{\tilde{C_2}(T_1)}{1- M_g - \left(\frac{5}{1-M_g}\right)^{p-1} \left(\frac{M_g C_{1-\a} T_1^\a}{\a}\right)^p}
 \]
 with
 \[
 \tilde{C}_2(T_1) = C_{T_1}  M_g+  \left(\frac{5}{1-M_g}\right)^{p-1} C_0^p [C_{T_1} + C_{T_1} M_g^p + 1].
 \]
 Repeating the arguments on $[kT_1, (k+1)T_1]$, we have
 \[
 \|u-v\|_{\B_{p,t}}^p \leq C_k(T_1) \int_{0}^{t} \beta(s) ds + L_k(T_1)\|\f-\psi\|_{C_h}^p,
 \]
 where
 \[
 C_k(T_1) = \frac{\tilde{C_k}(T_1)}{1- M_g - \left(\frac{5}{1-M_g}\right)^{p-1} \left(\frac{M_g C_{1-\a} T_1^\a}{\a}\right)^p}
 \]
 and $\tilde{C}_k(T_1)$ depends only on $C_i(T_1), i = 1,...,k-1$, and, therefore, is well defined for $T_1$ satisfying the conditions (\ref{cond1}) and (\ref{cond2}). Similarly, $L_k(T_1)$ depends only on $L_i(T_1), i = 1,...,k-1$, and, therefore, is also well defined for $T_1$ satisfying the conditions (\ref{cond1}) and (\ref{cond2}). Since the interval $[0,T]$ consists of $N$ intevals of length $T_1$ (with $T_1$ dependent only on $M_{g}$), (\ref{eq:3.3}) follows with $B = \max_{0 \leq k \leq N-1} C_k$. Finally, the inequality (\ref{eq:3.4}) follows from (\ref{eq:3.3}) by setting $a_2= b_2 =0$.
\end{proof}
{\bf Construction of the approximation sequence.} The existence and uniqueness of mild solution of (\ref{eq:1.1}) will be established using Picard type iteration scheme. For fixed $T>0$, let $u^{(0)}$ be a solution of (\ref{eq:3.1})  with $a \equiv b \equiv 0.$ We may now define
\begin{multline}\label{eq:3.18}
u^{(n)}(t) = S(t)(\f(0) + g(\f))  - g(u_t^{(n)})   - \int_0^t AS(t-s)g(u_s^{n}) ds  + \int_0^t S(t-s) f(u_s^{n-1}) ds  \\
+ \int_0^t S(t-s) \s(u_s^{n-1}) d W(s), \ t \in [0,T]
\end{multline}
with $u^{(n)}(t) = \f(t)$ for $t \in [-h,0]$. It was shown in \cite{24} that if the estimates  (\ref{eq:3.3}) and (\ref{eq:3.4}) hold, the iteration scheme (\ref{eq:3.18}) converges to the unique mild solution of (\ref{eq:1.1}). Hence the proof of Theorem \ref{Th:2.1} is complete in the case, when $\E \|\f\|^p_{C_h}<\infty.$ In general case, we may define a sequence of cut-off functions in spirit of \cite{13}, Th.7.4. 
\[
\f_n:=
\begin{cases}
\f, \ \text{ if } \|\f\|_{C_h} \leq n\\
0, \ \text{ if } \|\f\|_{C_h} > n.
\end{cases}
\]
Let $u_n$ be the solution of (\ref{eq:1.1}) with the initial condition $\f_n.$ By \cite{13}, Th.7.4, the sequence $u_n$ converges to the process $u(t)$, which solves (\ref{eq:1.1}). This completes the proof of Theorem \ref{Th:2.1}.

\subsection{Proof of Theorem \ref{Th:2.2}.} Using Lemma \ref{lem:3.2}, we have
\begin{equation*}\label{eq:3.19}
    \E \sup_{s \in [0,t]} \|u^{(n)}(s)\|^p \leq K_1 + C_1 \E \|\f\|^p_{C_h} + K \int_0^t \E \sup_{\t \in [0,s]} \|u^{(n-1)}(s)\|^p d\t.
\end{equation*}
An analogous estimate holds for the solution $u(t)$ as well. Hence, by Gronwall inequality we have
\begin{equation}\label{eq:3.20}
\E \sup_{s \in [0,t]} \|u(s)\|^p \leq (K_1+C_1 \E \|\f\|^p_{C_h}) e^{Kt}.
\end{equation}
Next, for two solutions $u(t)$ and $v(t)$ with random initial conditions $\f$ and $\psi$ respectively, using Lemma \ref{lem:3.2} we have
\begin{eqnarray*}\label{eq:3.21}
 & & \E \sup_{s \in [0,t]} \|u(s) - v(s)\|^{p} \leq L \E \|\f-\psi\|^p_{C_h} + B \int_0^t  \E N(\sup_{\t \in [0,s]} \|u_{\t} - v_\t\|^p) d s \\
\nonumber & & \leq L \E \|\f-\psi\|^p_{C_h} + B t N(\|\f-\psi\|^p_{C_h}) + B \int_0^t \E N(\sup_{\t \in [0,s]} \|u(\t) - v(\t)\|^p) ds,
\end{eqnarray*}
where we implicitly used the fact that $\sup_{\tau \in [-h,s]}  N(\|u(\t) - v(\t)\|^p)$ is attained either  for $\t \in [-h,0]$ or for $\t \in [0,s]$, and hence
\begin{multline*}
\E N(\sup_{\t \in [-h,s]} \|u(\t) - v(\t)\|^p) = \max\{\E N(\sup_{\t \in [-h,0]} \|u(\t) - v(\t)\|^p), \E N(\sup_{\t \in [0,s]} \|u(\t) - v(\t)\|^p) \} \\
\leq \E N(\sup_{\t \in [-h,0]} \|u(\t) - v(\t)\|^p) + \E N(\sup_{\t \in [0,s]} \|u(\t) - v(\t)\|^p).
\end{multline*}

It follows from (\ref{eq:3.20}) and that if $\E \|\f_n-\f\|^p_{C_h} \to 0, n \to \infty$, we have
\[
\limsup_{n \to \infty} \E \sup_{s \in [0,t]} \|u(s,\f_n) - u(s,\f)\|^p \leq B \int_0^t N(\limsup_{n \to \infty}\E \sup_{\t \in [0,s]}\|u(\t,\f_n) - u(\t,\f)\|^p) ds.
\]
Thus, from the condition {\bf (H2)(c)}, we have
\[
\limsup_{n \to \infty} \E \sup_{s \in [0,t]} \|u(s,\f_n) - u(s,\f)\|^p = 0,
\]
which completes the proof of Theorem \ref{Th:2.2}.

\subsection{Proof of Theorem \ref{Th:2.4}.} The idea of the proof will be based on Krylov-Bogoliubov compactness approach \cite{13}. Let $u(t)$ be the solution of (\ref{eq:1.1}), which satisfies the condition (\ref{eq:2.7}). Our goal is to show that the family of distributions $\mathcal{L}(u_t), t \geq T$, with $T\geq 2h$, is tight. To this end, we will show that for any $\ve>0$ there is a compact $K_{\ve} \subset C_h$, such that \begin{equation}\label{eq:3.22}
P\{u_t \in K_{\ve}\} > 1 -\ve.
\end{equation}
We have

\begin{eqnarray}\label{eq:3.23}
u_T = u(T+\theta) = S(T+\theta)(\f(0) + g(\f))  - g(u_{T+\theta})   - \int_0^{T +\theta} AS(T+ \theta -s)g(u_s) ds   \\ 
\nonumber + \int_0^{T+\theta} S(T+\theta-s) f(u_s) ds + \int_0^{T+\theta} S(T+ \theta -s) \sigma(u_s) d W(s).
\end{eqnarray}

\begin{lemma}\label{lem:3.3}
Let $\f(t,\omega):C_h \times \Omega \to C_h$ be a random process, such that for any $\gamma>0$
\begin{equation}\label{eq:3.24}
    \lim_{\sigma \to 0} P\{\omega: \sup_{t,s \in [-h,0], |t-s|< \sigma} \|\f(t,\omega) - \f(s,\omega)\|_{C_h} > \gamma\} = 0.
\end{equation}
Then for any $\ve>0$ there exists $B_{\ve} \in \Omega$ such that:\\
\begin{enumerate}
    \item $P\{B_\ve\} > 1 - \ve$;
    \item the family of functions $\{\f(t,\omega), \omega \in B_\ve\}$ is equicontinuous on $[-h,0]$.
\end{enumerate}
\end{lemma}
\begin{proof}
It follows from (\ref{eq:3.24}) that for any $\ve>0$ there exists $\{\delta_k, k \geq 1\}$ such that $\delta_k \to 0$ and 
\begin{equation}\label{eq:3.25}
    P\{\omega: \sup_{t,s \in [-h,0], |t-s|< \sigma} \|\f(t,\omega) - \f(s,\omega)\|_{C_h} > \frac{1}{k}\}\leq \frac{\ve}{2^k}.
\end{equation}
Denote 
\[
A_k:=\{\omega: \sup_{t,s \in [-h,0], |t-s|< \sigma} \|\f(t,\omega) - \f(s,\omega)\|_{C_h} > \frac{1}{k}\}.
\]
It follows from (\ref{eq:3.25}) that
\[
P\{\cup_{k \geq 1} A_k\} \leq \sum_{k=1}^\infty P(A_k) \leq \ve.
\]
Set 
\[
B:=\overline{\cup_{k \geq 1}A_k} = \cap_{k \geq 1}\overline{A_k} = \cap_{k \geq 1} \{\omega: \sup_{t,s \in [-h,0], |t-s|< \sigma} \|\f(t,\omega) - \f(s,\omega)\|_{C_h} \leq \frac{1}{k}\}.
\]
Then $P\{B\}>1-\ve$, and if $\omega \in B$, then $\omega \in \overline{A_k}$, and for any $\gamma>0$ choose $k > \frac{1}{\gamma}$, which would imply $\|\f(t,\omega) - \f(s,\omega)\|_{C_h} \leq \frac{1}{k} < \gamma$ for $|t-s|<\delta_k$. This completes the proof of Lemma \ref{lem:3.3}.
\end{proof}

\begin{lemma}\label{lem:3.4}
Let $D$ be a set of functions $\f_t \in C([-h,0], C_h)$ with the following properties:
\begin{enumerate}
    \item there is $R>0$ such that
    \begin{equation*}\label{eq:3.26}
    \|\f_0\|_{C_h} \leq R \text{ for any } \f_t \in D.
    \end{equation*}
    \item the set of functions $D$ is equicontinous on $[-h,0]$ in $t$.
\end{enumerate}
Then the set $\{g(\f_t), \f \in D\}$ is compact in $C_h$.
\end{lemma}
\begin{proof}
We will apply the infinite dimensional version of Arzela-Ascoli Theorem. To this end, we need to establish
\begin{itemize}
    \item[{[i]}] For fixed $t \in [-h,0]$, $\{g(\f_t), \f \in D\}$ is compact in $H$;
    \item[{[ii]}] for any $\ve > 0$, there is $\delta>0$, such that 
    \[
    \|g(\f_t) - g(\f_s)\| < \ve, \text{ if } |t-s|<\delta, \ \ t,s \in [-h,0].
    \]
\end{itemize}
It follows from the definition of $D$ that there is $C>0$ such that
\[
\sup_{t \in [-h,0]}\|\f_t\|_{C_h} \leq C \text{ for any } \f_t \in D.
\]
Using {\bf (H3)} we have
\[
\|g(\f_t)\|_{\a} \leq M_g \|\f_t\|_{C_h} + \|g(0)\| \leq M_g C + \|g(0)\|.
\]
This way Lemma \ref{L:2.1} implies that the set $D$ is bounded in $H_\a$ and therefore is compact in $H$.  Furthermore, for all $\ve>0$ we have
\[
\|g(\f_t)-g(\f_s)\| \leq \|g(\f_t)-g(\f_s)\|_{\a} \leq M_g \|\f_t-\f_s\|_{C_h},
\]
which completes the proof of Lemma. 
\end{proof}
\begin{lemma}\label{lem:3.5}
The operator 
\begin{equation*}\label{eq:3.27}
    [B \f](\theta):=\int_0^{T+\theta} A S(T+\theta -s) \f(s) ds
\end{equation*}
is a compact operator from $C([0,T],H_{\a})$ into $C_h$. 
\end{lemma}
\begin{proof}
We will also use the infinite dimensional analog of Arzela-Ascoli Theorem. For fixed $\theta \in [-h,0]$ and $\ve>0$ denote
\[
[B^\ve \f](\theta):=\int_0^{T+\theta-\ve} A S(T+\theta -s) \f(s) ds = \int_0^{T+\theta-\ve} A S(\ve) S(T+\theta -\ve -s) \f(s) ds,
\]
Since $A$ is a generator of analytic semigroup, we have $S(T+\theta -\ve -s) \f(s) \in D(A)$, then
\begin{equation*}\label{eq:3.27}
    \int_{0}^{T+\theta-\ve} A S(\ve) S(T+\theta -\ve -s) \f(s) ds = S(\ve) \int_{0}^{T+\theta-\ve} A S(T+\theta -\ve -s) \f(s) ds.
\end{equation*}
But
\begin{multline*}
 \left\| \int_{0}^{T+\theta-\ve} A S(T+\theta -\ve -s) \f(s) ds\right\| \leq \int_0^{T+\theta-\ve} C_{1-\a} (T+\theta-\ve-s)^{\a-1} ds \sup_{s \in [0,T]} \|\f(s)\|_{\a}\\
 \leq \frac{C_{1-\a} T^{\a}}{\a} \sup_{s \in [0,T]} \|\f(s)\|_{\a}.
\end{multline*}
Since $S(\ve)$ is compact, we conclude that the operators $B^{\ve}$ are compact. We have
\[
\|B^{\ve} \f - B\f\| \leq \int_{T+\theta-\ve}^{T+\theta} C_{1-\a} (T+\theta -s)^{\a-1} ds \sup_{s \in [0,T]} \|\f(s)\|_{\a} \to 0, \ \ve \to 0.
\]
Consequently $B^{\ve}$ converges to $B$ in the operator norm, thus $B$ is compact and the set
\[
\{B \f(\theta): \sup_{t \in [0,T]} \|\f(t)\|_{\alpha} \leq 1\}
\]
is compact in $H$. Fix $\theta$ and $r$, such that $-h \leq \theta \leq \theta + r \leq 0$. Then for $\f$ satisfying $\sup_{t \in [0,T]} \|\f(t)\|_{\a} \leq 1$ we have
\begin{multline*}\label{eq:3.28}
    \|(B \f(\theta + r) - B \f(\theta)\| \leq \left\|\int_0^{T+\theta}(A S(T+\theta +r -s) - AS(T+\theta -s))\f(s) ds \right\| \\
    +  \left\|\int_{T+\theta}^{T+\theta +r} A S(T+\theta +r -s) \f(s) ds \right\| := J_1 + J_2.
\end{multline*}
In view of Lemma \ref{L:2.2}, the estimate for $J_2$ is straightforward:
\[
J_2 \leq \int_{T+\theta}^{T+\theta +r} C_{1-\a} (T+\theta +r -s)^{\a-1} ds \sup_{t \in [0,T]} \| \f(t)\|_{\a} \to 0, r \to 0.
\]
In order to estimate $J_1$, we proceed as follows:
\[
J_1 \leq  \int_0^{T+\theta}\left\|(A S(s+r) - AS(s))\f(T+\theta - s)\right\| ds.
\]
Since $S(s) \f(T+\theta-s) \in D(A)$, using the continuity property of semigroup, we have
\begin{equation*}
(A S(s+r) - A S(s)) \f(T+\theta - s) = \\
S(r) A S(s)\f(T+\theta - s) - A S(s)\f(T+\theta - s) \to 0, r \to 0.
\end{equation*}
Furthermore, 
\[
\|A S(s+r)\f(s)\| \leq \frac{C_{1-\a}}{s^{1-\a}} \sup_{t \in [0,T]}\|\f(t)\|_{\alpha}.
\]
Hence, by Dominated Convergence Theorem, $J_1 \to 0$ as $r \to 0$, which competes the proof of Lemma \ref{lem:3.5}.
\end{proof}
The following result was shown in \cite{18}, Lemma 3.3:
\begin{lemma}\label{lem:3.6}
For any $p>2$ and $\beta \in (\frac{1}{p}, 1]$ the operator 
\begin{equation*}\label{eq:3.29}
    (G_\beta \f)(\theta) = \int_0^{T+\theta} (T+\theta -s)^{\beta-1} S(T+\theta -s) \f(s) ds
\end{equation*}
is a compact operator $L^p([0,T],H)$ to $C_h$.
\end{lemma}
For $\f \in C_h$, $\psi, \eta \in L^p([0,T], H)$, and $\xi \in C([0,T],H)$, introduce
\[
\mu[\f,\xi,\psi,\eta]:= S(T+\theta) (\f(0)+g(\f)) - g(u_t) + \int_{0}^{T+\theta} A  S(T + \theta -s) \xi(s) ds + (G_1 \psi)(\theta) + (G_{\beta} \eta)(\theta),
\]
and denote
\[
K(r):=\{\mu[\f,\xi,\psi,\eta] \in C_{h}: \|\f\|_{C_h} \leq r, \sup_{t \in [0,T]} \|\xi\| \leq r, \|\psi\|_{L^p([0,T], H)} \leq r, \|\eta\|_{L^p([0,T], H)} \leq r\}.
\]
It follows from Lemma \ref{lem:3.5} and Lemma \ref{lem:3.6} that $K(r)$ is compact in $C_h$. \\
Let $u(t,\f)$ be a solution of (\ref{eq:1.1}) with the initial condition $\f \in C_h$.
\begin{lemma}\label{lem:3.7}
Introduce 
\begin{multline*}
z(\theta):= S(T+\theta) (\f(0)+g(\f)) + \int_{0}^{T+\theta} A  S(T + \theta -s) g(u_s(\f)) ds + \\
\int_{0}^{T+\theta}  S(T + \theta -s) f(u_s(\f)) ds  + \int_{0}^{T+\theta}  S(T + \theta -s) \sigma(u_s(\f)) d W(s).
\end{multline*}
Assume the conditions of Theorem \ref{Th:2.1} hold. Then there is $C>0$ such that for any $r>0$ and for any $\f \in C_h$ with $\|\f\|_{C_r} \leq r$, we have
\begin{equation*}\label{eq:3.30}
P\{z(\theta) \in K(r)\} \geq 1 -  C r^{-p}(1+\|\f\|^p_{C_r}).
\end{equation*}
\end{lemma}
\begin{proof}
Fix $p>2$ and $\beta \in \left(\frac{1}{p}, \frac{1}{2}\right)$. Using the factorization formula \cite{25}, Th.5.2.5, we have
\begin{multline*}\label{eq:3.31}
z(\theta) = S(T+\theta) (\f(0)+g(\f)) + \int_{0}^{T+\theta} A  S(T + \theta -s) g(u_s(\f)) ds + \\
(G_1 f(u_s))(\theta)  + \frac{\sin(\beta \pi)}{\pi} (G_{\beta} Y(s))(\theta),
\end{multline*}
where
\[
Y(s) = \int_0^s (s-\t)^{-\beta} S(s-\t) \sigma(u_\t) d W(\tau).
\]
Lemma 7.2 \cite{9} yields
\begin{multline*}\label{eq:3.32}
\E  \int_{0}^{T}\|Y(s)\|^p ds = \E  \int_{0}^{T} \left\|  \int_0^s (s-\t)^{-\beta} S(s-\t) \sigma(u_\t) d W(\tau)\right\|^p \\
\leq C_0^p \E \int_{0}^{T} \left(\int_0^s (s-\t)^{-2\beta}\|\sigma(u_{\tau}\|^2_{\mathcal{L}_2^0} d\t \right)^{p/2} ds.
\end{multline*}
Using Hausdorff-Young inequality and the condition {\bf (H2)}(i), we get
\[
\E \int_0^{T} \|Y(s)\|^p ds \leq C_{21} \left(\int_0^T t^{-2\beta} dt\right)^{\frac{p}{2}} \left(1+ \E \sup_{t \in [0,T]} \|\f\|^p_{C_h}\right).
\]
Hence using (\ref{eq:3.20}) and Lemma \ref{L:2.4} we have
\begin{equation*}\label{eq:3.33}
\E \int_0^{T} \|Y(s)\|^p ds \leq C_{22} \left(1+ \|\f\|^p_{C_h}\right).
\end{equation*}
Similarly,
\begin{equation*}\label{eq:3.34}
\E \int_0^{T} \|f(u_s)\|^p ds \leq C_{23} \left(1+ \E \|\f\|^p_{C_h}\right).
\end{equation*}
Furthermore, using the condition {\bf (H3)}, (\ref{eq:3.20}) and Lemma \ref{L:2.4} we have
\begin{equation*}\label{eq:3.35}
\E \sup_{s \in [0,T]} \|g(u_s(\f))\|^p_{\a} \leq  C_{24} \left(1+ \|\f\|^p_{C_h}\right).
\end{equation*}
This way
\begin{multline*}
P\{z(\theta) \notin K(r)\} \leq P\{\|f(u_s(\f))\|_{L^p([0,T],H)} > r\} + P \{\|Y(s)\|_{L^p([0,T],H)} > \frac{\pi r}{\sin(\beta \pi)}\}\\
+ P\{\sup_{s \in [0,\t]} \|g(u_s(\f))\|_{\a} > r\} \leq r^{-p} C \left(1+ \E \|\f\|^p_{C_h}\right),
\end{multline*}
which completes the proof of Lemma \ref{lem:3.7}.
\end{proof}
\begin{lemma}\label{lem:3.8}
The random process $u_{T+\theta}$ satisfies (\ref{eq:3.24}) if $\E \|u_0\|^p_{C_h} < \infty$.
\end{lemma}
\begin{proof}
For any $\gamma>0$ Chebyshev's inequality yields 
\[
P\{\sup_{\theta_1, \theta_2 \in [-h,0]}\|u_{T+\theta_1} - u_{T+\theta_2}\|_{C_h}>\gamma\} \leq \frac{1}{\gamma^p} \E \sup_{\theta_1, \theta_2 \in [-h,0]}\|u_{T+\theta_1} - u_{T+\theta_2}\|_{C_h}^p.
\]
Let $\mu:=\theta_2 - \theta_1$. Then
\[
\sup_{\theta \in [-h,0]}\|u_{T+\theta+\mu} -u_{T+\theta}\|^p_{C_h} \leq \sup_{\theta \in [-h,0]} \sup_{s \in [T-h,T]} \|u(s+\theta+\mu) -u(s+\theta)\|^p \leq \sup_{s \in [0,T]}\|u(s+\mu) - u(s)\|^p,
\]
since $T>2h.$ So $u(s,\omega)$ is continuous in $H$ with probability 1 on $[0,T]$. Therefore $u(s)$ is uniformly continuous on $[0,T]$, so
\begin{equation*}\label{eq:3.36}
\sup_{s \in [0,T]}\|u(s+\mu) - u(s)\|^p \to 0 \text{ a.s. as } \mu \to 0. \end{equation*}
Since
\[
\E \sup_{s \in [0,T]}\|u(s+\mu)\|^p \leq \E \sup_{s \in [0,T+1]}\|u(s)\|^p < \infty,
\]
by Dominated Convergence Theorem, we obtain
\[
\E \sup_{s \in [0,T]}\|u(s+\mu) - u(s)\|^p \to 0 \text{ as } \mu \to 0,
\]
hence
\[
\E \sup_{\theta_1,\theta_2 \in [-h,0], |\theta_1-\theta_2|<\delta} \|u_{T+\theta_1}-u_{T+\theta_2}\|_{C_h}^p \to 0, \text{ as } \delta \to 0,
\]
which completes the proof of the Lemma.
\end{proof}
We now return to the proof of Theorem \ref{Th:2.4}. For any $\ve>0$ we will show that there is a compact set $K_\ve$ in $C_h$, which satisfies the condition (\ref{eq:3.22}). Using Lemma \ref{lem:3.8}, $u_{T+\theta}$ satisfies the condition (\ref{eq:3.24}), hence there exists a set $B_{\ve/3}$, such that $P\{B_\ve\}>1 - \frac{\ve}{3},$ and such that the set of functions 
\[
\{u_{T+\theta}(\omega),  \omega \in B_{\ve/3}\}
\]
is equicontinuous on $[-h,0]$. Next, choose $R_1>0$ such that
\begin{equation*}\label{eq:3.37}
    P\{\|u_T\|_{C_h}>R_1\} \leq  \frac{ \ve}{3}. 
\end{equation*}
Due to the condition (\ref{eq:2.7}), such choice is always possible. Therefore by Lemma \ref{lem:3.4} there is a compact set $K_{\frac{2\ve}{3}} \in C_h$ such that
\begin{equation}\label{eq:3.38}
  P\{g(u_{T+\theta})(\omega) \in K_{\frac{2\ve}{3}}\} \geq 1  - \frac{2 \ve}{3}. 
\end{equation}
Using the Markov property for $t>T$, for any Borel $K \subset C_h$ we have 
\begin{equation}\label{eq:3.39}
    P\{u_t \in K\} = \E(P(u_t \in K | \F_{t-1})) = \E(P(0,u_{t-T},T,K)) \geq \E(P(0,u_{t-T},T,K) \chi_{\{\|u_{t-T}\|^p_{C_h} \leq R_1\} }).
\end{equation}
 Now, using Lemma \ref{lem:3.7}, we can choose $r>R_1$ such that
\begin{equation}\label{eq:3.41}
 P\{z(\theta) \in K(r)\} \geq 1 - \frac{\ve}{3}.  
\end{equation}
It follows from (\ref{eq:3.23}), (\ref{eq:3.38}) and  \eqref{eq:3.41} there is a compact $K_{\ve} \subset C_h$ such that
\begin{equation*}\label{eq:3.42}
    P\{u_{T} \in K(\ve)\} \geq 1-\ve.
\end{equation*}
Setting $K=K_\ve$ in (\ref{eq:3.39}) we have
\[
P\{u_{t} \in K(\ve)\} \geq (1-\ve) P\{\|u_{t-T}\|^p_{C_h} \leq R_1\}.
\]
In view of the condition (\ref{eq:2.7}), the proof of Theorem \ref{Th:2.4} is complete. 

\section{Applications.}\label{sec4}  In this section we use the standard elliptic operator $A$ to illustrate our main result. We also provide explicit examples of non-Lipschitz nonlinearities $f$ and $\sigma$, for which the main results are applicable.

  Let $D$ be a bounded domain in $\R^d$ with $\partial D$ satisfying the Lyapunov condition, $H = L^2(D)$, and 
\begin{equation*}\label{eq:4.1}
    Au = \sum_{i,j=1}^{d} (a_{i,j}(x) u_{x_i})_{x_j} = {\rm{div}} (a(x) \nabla u).
\end{equation*}
Here $a_{i,j}$ are Holder continuous coefficients with the Holder exponent $\beta \in (0,1)$, symmetric, bounded, and satisfies the elliptic condition
\[
\sum_{i,j=1}^d a_{i,j} \eta_i \eta_j \geq C_0 \|\eta\|^2, \eta \in \R^d,
\]
for some $C_0>0.$ Let $e_n(x)$ be orthonormal basis in $H$, such that $e_n \in L^{\infty}(D)$, and $\sup_{n}\|e_n\|_{L^{\infty}(D)}<\infty.$ Introduce the covariance operator $Q \in \mathcal{L}(H)$ such that $Q$ is non-negative, $Tr(Q)< \infty$ and $Q e_n = \lambda_n e_n$. This enables us to define
\[
W(t):= \sum_{i=1}^{\infty} \sqrt{\lambda_i} \beta_i(t) e_i(x), t \geq 0,
\]
which is a $Q$-Wiener process for $t \geq 0$ with values in $L^2(D)$. Denote $U := Q^{\frac{1}{2}}(L^2(D))$. It follows from \cite{26}, Lemma 2.2, that $U \in L^{\infty}(D)$. Following \cite{26}, introduce the multiplication operator $\Psi:U \to H$ as follows. For fixed $\f \in L^2(D)$, let $\Psi(\f) = \f \cdot \psi$ for $\psi \in U$. Since $\f \in L^2(D)$ and $\psi \in L^{\infty}(D)$, the operator $\Psi$ is well defined, and hence $\Psi \circ Q^{1/2} : L^2(D) \to L^2(D)$ defines a Hilbert-Schmidt operator, with
\[
\|\Psi \circ Q^{1/2}\|^2_{\L} \leq Tr(Q) \sup_{n} \|e_n\|^2_{\infty} \|\f\|^2_{L^2(D)}.
\]
Our main object of interest in this section is the following delay equation:
\begin{equation}\label{eq:4.2}
    d[u(t,x) + \int_{D}b(x,u(t-h,y),y)dy] = [{\rm div}(a(x) \nabla u(t,x)) + f(u(t-h,x))] dt + \sigma(u(t-h,x)) dW(t) 
\end{equation}
for $t>0$, with $u(t,x)=\f(t,x)$ for $t \in [-h,0]$ and $u(t,x) = 0$ for $x \in \partial D, t \geq 0.$ Here $b(x,z,y):\R^d \times \R^1 \times \R^d \to \R$, $f: \R \to \R$ and $\sigma: \R \to \R.$ Introduce the mapping 
\[
g(\f)(x):= \int_{D}b(x,\f(-h),y) dy
\]
as a map from $C_h$ to $L^2(D)$. Then the problem (\ref{eq:4.2}) can be viewed in the abstract form (\ref{eq:1.1}), with $D(A) = H^2(D) \cap H_0^1(D)$.

\begin{corollary}
\begin{itemize}
\item[{\bf[i]}] Suppose $f$ and $\sigma$  satisfy {\bf(H2)} (as Nemytskii maps), and the initial condition $\varphi$ satisfies {\bf(H4)}. In addition, assume the function $b$ is continuous with respect to all of its variables, and there exist constants $A>0$ and $M_g>0$ such that 
\[
|b(x,0,y)|+|\nabla_x b(x,z,y)| \leq A 
\]
and
\[
|b(x,z_1,y) - b(x,z_2,y)| + |\nabla_x b(x,z_1,y) - \nabla_x b(x,z_2,y)| \leq M_g |z_1-z_2|
\]
for all $x,y \in D$, $z_1,z_2 \in \R$ and \begin{equation}
2 M_g^2 \ {\rm meas}^2(D)<1.
\end{equation}
Then \eqref{eq:4.2} has a unique mild solution. 
\item[{\bf[ii]}] If, in addition, the functions $f$, $\sigma$ and $b$  are bounded, there is an invariant measure for \eqref{eq:4.2}.
\end{itemize}

\end{corollary}

\begin{proof}
Let us verify that the conditions of Theorem \ref{Th:2.1} hold.  Since $(-A)$ is a sectorial self-adjoint operator, it follows, e.g. from \cite{27}, p.335 that $\sigma(-A) > \sigma_0>0$, and $(-A)^{-1}$ is compact. Therefore the assumption {\bf (H1)} holds. Furthermore, following \cite{28} introduce the interpolation space $D_A(\frac{1}{2}, 2) = H_0^1$. Therefore by Proposition A.17 \cite{19} $D_A(\frac{1}{2}, 2)$ is isomorphic to $D((-A)^{\frac{1}{2}})$, and
\begin{equation*}\label{eq:4.3}
\|g(\f)\|^2_{\frac{1}{2}} = \|g(\f)\|^2_{H_0^1} = \int_{D} |(g  (\f)(x)|^2 dx + \int_{D} |\nabla g(\f)(x)|^2 dx.
\end{equation*}
But 
\begin{equation*}\label{eq:4.4}
\int_{D} |g(\f)(x)|^2 dx = \int_{D} dx \left(\int_{D} b(x,\f(-h,y),y) dy\right)^2 \leq L^2 {\rm meas}^2(D)\left( \int_{D} |\f(-h,y)|^2 dy +1 \right)<\infty.
\end{equation*}
By Lebesgue Theorem on differentiation of integrals,
\[
\nabla g(\f)(x) = \int_{D} \nabla_x b(x,\f(-h,y),y) dy,
\]
hence
\[
\int_{D}|\nabla g(\f)(x)|^2 dx < \infty.
\]
Finally, 
\begin{eqnarray*}
 & & \|g(\f_1) - g(\f_2)\|^2_{\frac{1}{2}} = \int_{D} |g(\f_1)(x)-g(\f_2)(x)|^2 dx + \int_{D} |\nabla_x g(\f_1)(x)- \nabla_x g(\f_2)(x)|^2 dx\\
& &\leq \int_{D} dx \left(\int_{D}|b(x,\f_1(-h,y),y) - b(x,\f_2(-h,y),y)| dy \right)^2 \\
& & + \int_{D} dx \left(\int_{D}|\nabla_x b(x,\f_1(-h,y),y) - \nabla_x b(x,\f_2(-h,y),y)| dy \right)^2\\
& & \leq 2 {\rm meas}^2(D) M_g^2 \sup_{\theta \in [-h,0]} \int_{D} |\f_1(\theta,y) - \f_2(\theta,y)|^2 dy = 2 {\rm meas}^2(D) M_g^2  \|\f_1 - \f_2\|^2_{C_h},
\end{eqnarray*}
which implies that the condition {\bf (H3)} holds. Thus {\bf[i]} follows by Theorem  \ref{Th:2.1}. \\

In order to derive {\bf[ii]}, it remains to verify the condition (\ref{eq:2.7}) and apply Theorem \ref{Th:2.4}.  To this end, the spectral condition on $A$ implies the exponential contraction property of the semigroup 
\[
\|S(t)\| \leq K e^{-\sigma_0 t}, t \geq 0.
\]
Assume
\[
|f|^2 + |\sigma|^2 + |b|^2 + |\nabla_x b|^2 \leq C.
\]
Then by  Lemma \ref{L:2.4}
\[
\sup_{t \geq 0} \E \|u_t\|^2_{C_h} \leq \E \|\f\|^2_{C_h} + \sup_{t \geq 0}\E \|u(t)\|^2_{L^2(D)},
\]
which implies 
\begin{eqnarray*}
& & \E \|u(t)\|^2 \leq 5 \E \|S(t)(\f(0)+g(\f))\|^2 + 5 \E\|g(u_t)\|^2 + 5 \E \left\|\int_0^t A S(t-s)g(u_s) ds\right\|^2 \\
& &+ 5 \E \left\|\int_0^t S(t-s)f(u_s) ds\right\|^2 + 5 \E \left\|\int_0^t S(t-s)\sigma(u_s) d W(s)\right\|^2 \leq 5 K^2 e^{-2\sigma_0 t}  \E \|\f(0)+g(\f)\|^2 \\
& & C^2 {\rm meas}^2(D) + \E \left(\int_0^t \left\| A^{1/2} S(t-s)\| \|g(u_s)\|_{1/2}\right\| ds\right)^2  \\
& & K^2 \int_0^t e^{-\sigma_0(t-s)} C ds + K^2\sum_{k=1}^{\infty} \lambda_k \sup_{n\geq 1}\|e_n\|^2_{\infty} \int_0^t e^{-\sigma_0(t-s)}  ds.
\end{eqnarray*}
Since $ \|g(u_s)\|_{1/2} < \infty$ and $\|A^{1/2} S(t-s)\|\leq C_{1/2}(t-s)^{-\frac{1}{2}}e^{-\sigma_0(t-s)}$, we have $\sup_{t \geq 0}\E \|u(t)\|^2 <\infty$. Therefore, the condition (\ref{eq:2.7}) holds by Chebyshev's inequality. 
\end{proof}
 \begin{example}
Suppose $meas(D) = 1$. For $x \geq 0$ let
\begin{equation}\label{choice of N}
N(x) :=
\begin{cases}
0, \ x =0;\\
-x \ln(x), \ 0<x \leq e^{-2};\\
x + e^{-2}, \ x > e^{-2}.
\end{cases}
\end{equation}
In view of Remark \ref{Rem1}, this choice of $N$ satisfies the conditions (a) and (c) in {\bf(H2)}. For $p>2$ let
\begin{equation}\label{choice of f}
f(x) :=
\begin{cases}
0, \ x =0;\\
|x| p^{1/p} \left|\ln|x|\right|^{1/p}, \ 0<|x| \leq e^{-2};\\
\text{ any bounded Lipschitz function with } L \leq 1, \, |x| > e^{-2}.
\end{cases}
\end{equation}
We first claim that for all $x, y \in [0, e^{-2}]$ we have
\begin{equation}\label{difference}
|f(x) - f(y)|^p  \leq N(|x-y|^p)
\end{equation}
Indeed, since $f$ is concave on $[0, e^{-2}]$ and $f(0) = 0$, it is also sub-additive on this interval, i.e. $\forall a, b: f(a) + f(b) \geq f(a+b)$. The latter inequality means that if $0<x<y<e^{-2}$, we have
\[
f(x) + f(y-x) \geq f(y),
\]
or 
\[
(f(y)-f(x))^p \leq f^p(y-x) = N\left((y-x)^p\right).
\]
Using the same reasoning if $y<x$, the bound (\ref{difference}) follows. Since $f$ is even on $[-e^{-2}, e^{-2}]$ and Lipschitz outside of this interval, it is easy to see that in fact \eqref{difference} holds for all $x, y \in \R$. Note that since $meas(D) = 1$, the bound (\ref{difference}) is equivalent to the inequality
\begin{equation}\label{N}
\|f(\varphi_1) - f(\varphi_2)\|^p  \leq N(\|\varphi_1 - \varphi_2\|^p)
\end{equation}
which, in turn, yields {\bf(H2)} (b) for two {\bf constant} initial conditions $\varphi_{1}(t,x) \equiv\varphi_{1}$ and $\varphi_{2}(t,x) \equiv \varphi_{2}$, $x \in D$, $t \in [-h,0]$.  In order to obtain the estimate  {\bf(H2)} (b) for non-constant $\varphi_{1}$ and $\varphi_{2}$, we equi-partition $D$ into subdomains, approximate $\varphi_{1}$ and $\varphi_{2}$ with constants in each subdomain, and use \eqref{difference} in conjunction with the concavity of $N$ to get \eqref{N} for piecewise constant functions. The desired result will then follow if we take the limit as the number of partitions goes to infinity. \\

Thus, the equation \eqref{eq:4.2} with a non-Lipschitz $f$, given by \eqref{choice of f}, is an example of a well-posed neutral type stochastic functional-differential equation, which admits the existence of invariant measure. The same conclusion will hold if we choose $\sigma \equiv f$, given by \eqref{choice of f}. The only difference in this case that we need to slightly modify the choice of $N$ in \eqref{choice of N} by taking $\tilde{N}:= \sum_{k=1}^{\infty} \lambda_k \sup_{n}\|e_n\|_{\infty}^2 N$.
\end{example}

\section*{Acknowledgement} The research of Oleksandr Misiats was supported by Simons Collaboration
Grant for Mathematicians No. 854856. The work of Oleksandr Stanzhytskyi was partially supported by the National Research Foundation of Ukraine No. F81/41743 and Ukrainian Government Scientific Research Grant No. 210BF38-01.

\bibliography{bibliogratyit}
\end{document}